\documentclass[12pt]{amsart}
\usepackage{amsmath}
\usepackage{amsmath,mathtools}
\usepackage{imakeidx}
\makeindex

\usepackage{hyperref}
\usepackage{graphicx, wrapfig}
\usepackage{txfonts}
\usepackage{ulem, cancel}
\usepackage[usenames]{color}

\usepackage{comment}

\usepackage{enumitem}
\usepackage{xcolor}

\counterwithin{figure}{section}
\title[Gromov's Compactness for Intrinsic Timed-Hausdorff]{Gromov's Compactness Theorem for the Intrinsic Timed-Hausdorff Distance}

\author[Che]{Mauricio Che} 
\author[Perales]{Raquel Perales}
\author[Sormani]{Christina Sormani}

\thanks{R. Perales was funded by 
the Austrian Science Fund
(FWF) [Grant DOI: 10.55776/EFP6]. M.~Che was funded by the Austrian Science Fund (FWF) [Grant DOI: 10.55776/STA32]. 
  The research was also funded in part by Sormani's PSC-CUNY and NSF DMS-1612409 grants. 
}

\theoremstyle{plain} 

\newtheorem{thm}{Theorem}[section]
\newcommand{\bt}{\begin{thm}}
\newcommand{\et}{\end{thm}}

\newtheorem{cor}[thm]{Corollary}

\newcommand{\bc}{\begin{cor}}
\newcommand{\ec}{\end{cor}}

\newtheorem{lem}[thm]{Lemma}   

\newcommand{\bl}{\begin{lem}}
\newcommand{\el}{\end{lem}}

\newtheorem{prop}[thm]{Proposition}
\newcommand{\bp}{\begin{prop}}
\newcommand{\ep}{\end{prop}}

\newtheorem{defn}[thm]{Definition}
\newtheorem{conj}[thm]{Conjecture}

\newcommand{\ben}{\begin{itemize}}
\newcommand{\een}{\end{itemize}}

\newcommand{\bd}{\begin{defn}}       
\newcommand{\ed}{\end{defn}}

\newtheorem{rmrk}[thm]{Remark}   

\newcommand{\br}{\begin{rmrk}}
\newcommand{\er}{\end{rmrk}}

\newcommand{\GHto}{\stackrel { \textrm{GH}}{\longrightarrow} }

\newcommand{\tHto}{\stackrel { \tau-\textrm{H}}{\longrightarrow} }

\newcommand{\be}{\begin{equation}}

\newcommand{\ee}{\end{equation}}

\newcommand{\diam}{\operatorname{diam}}

\newcommand{\disjointunion}{\sqcup}

\def\iff{\Longleftrightarrow}
\def\implies{\Longrightarrow}

\begin{document}

\begin{abstract}  
The intrinsic timed-Hausdorff distance between timed-metric-spaces, first introduced by Sakovich--Sormani, yields a weak notion of convergence for space-times.
In this paper we prove a compactness theorem for the intrinsic timed-Hausdorff convergence of timed-metric-spaces using timed-Fr\'echet maps. Our proof introduces the notion of ``addresses'' and provides a new way of stating Gromov's original compactness theorem for Gromov--Hausdorff  convergence of metric spaces. We also obtain a new Arzel\`a--Ascoli theorem for real valued uniformly bounded Lipschitz functions on Gromov--Hausdorff converging compact metric spaces. 
Moreover, we establish
the triangle inequality for the intrinsic timed-Hausdorff distance. 
\end{abstract}

\maketitle
\tableofcontents

\section {\bf Introduction}

One of the most profound theorems in metric geometry is Gromov's Compactness Theorem, which states that  a sequence of equibounded and equicompact metric spaces has a subsequence converging in the Gromov--Hausdorff sense to a compact metric space \cite{Gromov-poly,Gromov-1981}. Moreover, when a sequence of metric spaces converges in the Gromov--Hausdorff sense 
one can prove Arzel\`a--Ascoli theorems for sequences of uniformly bounded Lipchitz
functions on the sequence of spaces (c.f. Greene--Petersen \cite{Greene-Petersen}).

A sequence of compact Riemannian manifolds can be canonically converted into compact metric spaces using the Riemannian distance. With Gromov--Hausdorff convergence, 
one can study sequences of Riemannian manifolds with different topologies and even different dimensions. With uniform upper diameter bounds and lower Ricci curvature bounds, Gromov proved that Riemannian manifolds are equicompact and equibounded and thus any sequence subconverges to a limit space \cite{Gromov-1981,Gromov-metric}.  This was applied with great success by Cheeger, Colding, Fukaya, Naber, Sormani, Wei and others in many papers including \cite{ChCo-almost-rigidity,ChCo-PartI,CheegerGromovI,Cheeger-Naber-Invent-2013,Fukaya-87,SorWei3}.

One would also like to study sequences of smooth Lorentzian manifolds which are not diffeomorphic.  In \cite{SakSor-Notions}, Sakovich and Sormani introduced the notion of the intrinsic timed-Hausdorff convergence as one possible way to achieve this.  They described a method of converting smooth causally null compactifiable space-times into compact timed-metric-spaces in a canonical way using the cosmological time function of Anderson--Galloway--Howard \cite{AGH} and the null distance of Sormani--Vega \cite{SV-Null}. 
In general, timed-metric-spaces, $(X_j,d_j,\tau_j)$, are metric spaces, $(X_j,d_j)$, endowed with a $1$-Lipschitz function, $\tau_j\colon X_j\to \mathbb R$, and a causal structure. They then defined the intrinsic timed-Hausdorff convergence of these timed-metric-spaces.

In this paper we prove a compactness theorem for the notion of intrinsic timed-Hausdorff convergence. For the full, somewhat stronger statement, see Theorem~\ref{thm:timed-Gromov-compactness-A}. Note that no knowledge of Lorentzian geometry is required throughout this manuscript, but 
further applications of our compactness theorem to Lorentzian manifolds will appear in future work.

\begin{thm}[Timed Gromov Compactness Theorem--short version]\label{thm:timed-Gromov-compactness}
If $(X_j,d_j,\tau_j)$ is a sequence of compact timed-metric-spaces that are equibounded, 
\be
\exists D>0 \,s.t.\, \diam_{d_j}(X_j)\le D,
\ee
equicompact,
\be\label{eq:equicompact-A}
\forall R\in (0,D]\,
\exists N(R) \in \mathbb N \,s.t.\,
\exists \{x^j_{R,i}:i=1,\ldots,N(R)\}\,\, s.t.\,\,
X_j\subset \bigcup_{i=1}^{N(R)} B(x^j_{R,i},R)
\ee
and have a uniform bound on their $1$-Lipschitz functions 
\be
\tau_j(X_j)\subset [0,\tau_{max}],
\ee
then a subsequence converges in the intrinsic timed-Hausdorff sense to a compact timed-metric-space $(X_\infty, d_\infty, \tau_\infty)$. 
In fact, there exist distance and time preserving 
timed-Fr\'echet maps,
\be
\varphi_j\colon X_j \to [0,\tau_{max}]\times Z \subset \ell^\infty,
\ee
where $Z$ is compact, such that 
\be \label{eq:H-A}
d_H^{[0,\tau_{max}]\times Z}(\varphi_j(X_j),\varphi_\infty(X_\infty))\to 0.
\ee
\end{thm}

We remark that there are special classes of space-times where the cosmological time functions are easy to control.  These include the big-bang space-times and future developed space-times.  Sormani--Vega \cite{SV-BigBang} and Sakovich--Sormani \cite{SakSor-Notions} defined convergence for such space-times based on the notions of pointed GH convergence \cite{Gromov-metric} and GH convergence of metric pairs by Che, Galaz-Garc\'ia, Guijarro,
and Membrillo Solis \cite{CGGGMS}.
Such notions already have compactness theorems proven by Gromov \cite{Gromov-metric} and by
Che and Ahumada G\'omez \cite{Che-Gomez-pairs} respectively. The intrinsic timed-Hausdorff distance applies to a far wider class of space-times including sub-space-times that exhaust asymptotically flat space-times \cite{SakSor-Notions}.

It should be noted that some authors have used the Sormani--Vega null distance of \cite{SV-Null} to convert space-times into metric spaces without keeping track of the time functions. See the work of Allen, Burtscher, Garc\'ia-Heveling,
Kunzinger, and Steinbauer
in \cite{Allen-Null, Allen-Burtscher-22, Burtscher-Garcia-Heveling-Global, Kunzinger-Steinbauer-22}.
Furthermore, Allen, Burtscher,
Kunzinger, and Steinbauer \cite{Allen-Null, Allen-Burtscher-22, Kunzinger-Steinbauer-22} have applied Gromov's original compactness theorem to obtain a GH limit space which is only a metric space with no causal structure unlike our timed-metric-spaces.

Other notions of weak convergence of space-times do not involve the conversion of the space-times into metric spaces or timed-metric-spaces.   Minguzzi--Suhr \cite{Minguzzi-Suhr-24} have a compactness result in the setting of ``Lorentzian metric spaces" (which are not metric spaces) building upon work of  Noldus \cite{Noldus-limit,Noldus}. See also the
work of Mondino--S\"amann
 \cite{Mondino-Saemann-2025} in the setting of ``Lorentzian pre-length spaces'', as introduced in \cite{Kunzinger-Saemann}, where they define a notion of Lorentzian Gromov--Hausdorff convergence and obtain a precompactness result using causal diamonds and Lorentzian distances.  
 It is not yet clear how their notions of convergence relate to the Sakovich--Sormani notions defined using timed-metric-spaces in \cite{SakSor-Notions}.  This is worth exploring further.

We note that, in upcoming work, Sakovich--Sormani will be introducing timed intrinsic flat convergence of space-times whose limit spaces are rectifiable timed-metric-spaces with biLipchitz charts \cite{SakSor-SIF}.   Our compactness theorem proven in this paper will be applied in that one in an essential way.

The structure of the paper is as follows. In Section~\ref{sec:Background} we review preliminary notions and results about the Gromov--Hausdorff distance, the classical Gromov's compactness theorem, and the intrinsic timed-Hausdorff distance. Our compactness theorem is proven in Sections~\ref{sect:p1}-~\ref{sect:p3}.  In Section~\ref{sect:p1} we begin the proof by imitating Gromov's proof of his compactness theorem in \cite{Gromov-poly}.  As we added significant details to his original proof, we developed the notion of ``addresses" for points in Section~\ref{sect:p2}.   With these ``addresses'' we complete a stronger statement of the compactness theorem which controls the convergence of time in a uniform way in Section~\ref{sect:p3}.

Our notion of ``addresses" provides a new way of stating Gromov's Compactness Theorem for sequences
of metric spaces, that we state and prove in Section~\ref{sect:ap1}.  We also restate an Arzel\`a--Ascoli Theorem for sequences of real valued uniformly bounded Lipschitz functions on a converging sequence of metric spaces using ``addresses'' in 
Section~\ref{sect:ap2}.   We suggest a more general Arzel\`a--Ascoli Conjecture that could be proven using ``addresses'' in 
Section~\ref{sect:ap3}.
We finalize the paper establishing in the Appendix \ref{sec-appendix}
the triangle inequality for the intrinsic timed Hausdorff distance.

{\bf Acknowledgements:} We would like to thank Anna Sakovich (Uppsala University) for online conversations in July 2025 that lead us to pursuing this project.
We all gratefully acknowledge support from the Simons Center for Geometry and Physics and Stony Brook University where we had the opportunity to meet together in person in September 2025 as part of the Program {\em Geometry and Convergence in Mathematical General Relativity}. We would also like to thank Davide Carazzato (University of Vienna) for sharing ideas on the proof of the triangle inequality for the intrinsic timed-Hausdorff distance.
 R. Perales was funded by 
the Austrian Science Fund
(FWF) [Grant DOI: 10.55776/EFP6]. M.~Che was funded by the Austrian Science Fund (FWF) [Grant DOI: 10.55776/STA32]. C. Sormani was funded in part by a PSC-CUNY grant. 
  For open access purposes, the authors have applied a CC BY public copyright license to any author-accepted manuscript version arising from this submission.

\section{\bf Background}\label{sec:Background}

\subsection{\bf Review of the Hausdorff Distance}
\label{sect:GH1}

The {\bf Hausdorff distance} between subsets, 
$U,W\subset Z$, lying in a common metric space, $(Z,d_Z)$,
is defined by
\be\label{eq:defn-Hausdorff}
d_H^Z\Big(U,W\Big)=\inf
\left\{ r>0:  
\begin{array}{c} \forall u\in U, \exists w\in W 
\textrm{ s.t. }
d_Z(u,w)<r\\
\forall w\in W, \exists u\in U
\textrm{ s.t. }
d_Z(u,w)<r
\end{array}
\right\}.
\ee

Recall that the Hausdorff distance between sets, $U,W\subset Z$, is extrinsic in the sense that it depends on the distance in   the extrinsic space.

\subsection{\bf Review of Gromov--Hausdorff Distance}
\label{sect:GH2}

Gromov defined an intrinsic Hausdorff distance between
metric spaces, $(X_j,d_j)$, which do not lie in a common metric space in \cite{Gromov-metric}.  This is now called
the Gromov--Hausdorff (GH) distance. The GH distance is intrinsic in the sense that it depends only on the intrinsic geometry of the two given metric spaces and not the extrinsic geometry of some common ambient space.  

In order to define the Gromov-Hausdorff distance between two compact metric spaces, $(X_j,d_j)$, consider all possible common metric spaces, $(Z,d_Z)$, and all maps:
$
\varphi_j\colon X_j\to Z
$
that are distance preserving, 
\be
d_Z(\varphi_j(p),\varphi_j(q))=d_j(p,q) \qquad \forall p,q \in X_j,
\ee
and take the following infimum:
\be
d_{GH}\bigg((X_1,d_1),(X_2,d_2)\bigg)=\inf d_H^Z\bigg(\varphi_1(X_1),\varphi_2(X_2)\bigg).
\ee

Thanks to the infimum taken over all possible distance preserving maps, the Gromov--Hausdorff distance depends only on the intrinsic geometry of the original pair of metric spaces.  Yet at the same time the GH distance between metric spaces has many of the same properties as the Hausdorff distance.   In particular, there are correspondences between $X_1$ and $X_2$ with controlled distance distortion that are often used to define the GH distance, but we will not use correspondences in this paper.

\subsection{\bf Review of Gromov's Compactness Theorem}\label{sect:GH3}

Recall Gromov's Compactness Theorem
proven in \cite{Gromov-metric} and \cite{Gromov-poly}:

\begin{thm}[Gromov Compactness Theorem]\label{thm:Gromov-compactness}
If $(X_j,d_j)$ is a sequence of compact metric spaces that are equibounded, 
\be\label{eq:equibounded}
\exists D>0 \,s.t.\, \diam_{d_j}(X_j)\le D,
\ee
equicompact,
\be\label{eq:equicompact}
\forall R\in (0,D]\,
\,\exists N(R) \in \mathbb N \,s.t.\,
\exists x^j_{R,i}\in X_j\,\, s.t.\,\,
X_j\subset \bigcup_{i=1}^{N(R)} B(x^j_{R,i},R),
\ee
then a subsequence converges in the Gromov--Hausdorff sense to a compact metric space $(X_\infty, d_\infty)$:
\be
d_{GH}((X_j,d_j),(X_\infty,d_\infty))\to 0.
\ee
In fact, there exist a
compact metric space, $(Z,d_Z)$, 
and distance  preserving maps,
\be
\varphi_j\colon X_j \to  Z,
\ee
such that the Hausdorff distance between the images converges to zero:
\be
d_H^{Z}\bigg(\varphi_j(X_j),\varphi_\infty(X_\infty)\bigg)\to 0.
\ee
\end{thm}

\subsection{\bf Review of Fr\'echet Maps}\label{sect:back-Frechet}\label{sect:F1}

Here we review the notion of a Fr\'echet map \cite{Frechet1910}. \footnote{In v1 of \cite{SakSor-Notions} this was accidentally attributed to Kuratowski. The actual definition of a Kuratowski map can be consulted here \cite{Kuratowski1935}.}

First recall that a separable metric space, $(X,d)$,  is a metric space containing a countable dense collection of points.

Recall also the Banach space  
\be\label{eq:defn-ell-infty}
\ell^\infty=\{(s_1,s_2,...)\,: s_i \in {\mathbb{R}}, \, d_{\ell^\infty}((s_1,s_2,...),(0,0,...))<\infty\}
\ee
where   
\be\label{eq:d-ell-infty}
d_{\ell^\infty}((s_1,s_2,...),(r_1,r_2,...))
=\sup\{|s_i-r_i|\,:\, i\in {\mathbb N}\}.
\ee

\begin{defn}\label{defn:Frechet}
Given a separable and bounded metric space, $(X,d)$, 
for any countable dense collection of points $\mathcal{N}=\{x_1,x_2,...\}$ of $X$, 
one can define a Fr\'echet map
\be
\kappa_X=\kappa_{X,\mathcal{N}}\colon (X,d)\to (\ell^\infty, d_{\ell^\infty})
\ee
by
\be\label{eq:Frechet}
\kappa_X(x)=(d(x_1,x),d(x_2,x),\ldots)\subset \ell^\infty.
\ee
\end{defn}

We remark that one can define Fr\'echet maps for unbounded metric spaces by fixing a reference point $x_0\in X$ and setting
\begin{equation}
\kappa_X(x) = (d(x_1,x)-d(x_1,x_0),d(x_2,x)-d(x_2,x_0),\ldots) \in \ell^\infty.
\end{equation}
However, since our main result, Theorem~\ref{thm:timed-Gromov-compactness}, concerns \textit{compact} timed-metric-spaces, Definition~\ref{defn:Frechet} suffices for our purposes. 

Now we state Fr\'echet's embedding theorem \cite{Frechet}:

\begin{thm}[Fr\'echet]
\label{thm:K-dist-pres}
Fr\'echet maps as in Definition~\ref{defn:Frechet}
are distance preserving:
\be\label{eq:K-dist-pres}
d_{\ell^\infty}(\kappa_X(x),\kappa_X(y))=
d(x,y) \quad \forall \, x,y\in X.
\ee
\end{thm}

It is well known that
one can define a version of the Gromov--Hausdorff distance by taking the infimum of the Hausdorff distance between images of Fr\'echet maps into
$\ell^\infty$.

\begin{defn}\label{defn:kappa-GH}
The Fr\'echet-GH distance between two metric spaces is defined as
\be
d_{\kappa-GH}((X,d_X),(Y,d_Y))
=\inf d_H^{\ell^\infty}(\kappa_X(X),\kappa_Y(Y))
\ee
where the infimum is over all pairs of Fr\'echet maps
\be
\kappa_X\colon X \to \ell^\infty
\quad\textrm{ and } \quad
\kappa_Y\colon Y \to \ell^\infty,
\ee
that is, the infimum is taken over
all selections of countably dense points and reorderings of these selected points in
$X$ and in $Y$.
\end{defn}

By Fr\'echet's embedding theorem, for any pair of compact metric-spaces one immediately has
\be\label{eq:kappa-ge}
d_{\kappa-GH}((X,d_X),(Y,d_Y))\ge
d_{GH}((X,d_X),(Y,d_Y)).
\ee
In fact, this notion is biLipschitz equivalent to
the Gromov--Hausdorff distance:
\be \label{eq:kappa-GH-comp}
d_{GH}((X,d_X),(Y,d_Y))
\le d_{\kappa-GH}((X,d_X),(Y,d_Y))\le
2 d_{GH}((X,d_X),(Y,d_Y)).
\ee
This close relationship with the Gromov--Hausdorff distance
is discussed and reproven in the background of \cite{SakSor-Notions} and
directly inspired their definition of the intrinsic
timed-Hausdorff distance using timed-Fr\'echet maps.

\subsection{\bf Review of Timed-Metric-Spaces}
\label{sect:SS1}
The notion of a timed-metric-space was introduced by Sakovich--Sormani in \cite{SakSor-Notions} because they had constructed a canonical way to convert certain classes of smooth Lorentzian space-times into timed-metric-spaces keeping track of the causal structure of the original Lorentzian manifold.   For our purposes, we need only the abstract notion of a timed-metric-space:

\begin{defn}\label{defn:timed-metric-space}
A {\bf timed-metric-space}, 
$(X,d,\tau)$, is a metric space, $(X,d)$ \footnote{$
d\colon X\times X\to [0,\infty)$ is a definite, symmetric map satisfying the triangle inequality}, 
endowed with a $1$-Lipschitz time function,  
$
\tau\colon X\to [0,\infty),
$
and a causal structure defined as follows: 
\be\label{eq:causal}
\forall q \in X\quad p\in J_X^+(q) \iff \tau(p)-\tau(q)=d(p,q).
\ee
\end{defn}

Given a map $F\colon(X,d_X,\tau_X)\to (Y,d_Y,\tau_Y)$ 
we say that it is {\bf distance preserving} iff
\be
d_Y(F(p),F(q))=d_X(p,q) \qquad \forall p,q\in X.
\ee
We say that $F$ is {\bf time preserving} iff
\be
\tau_Y(F(p))=\tau_X(p) \qquad \forall p\in X.
\ee
We say that $F$ {\bf preserves the causal structure} iff
\be
F(p)\in J^+_Y(F(q)) \iff p\in J^+_X(q)
\ee
which happens when $F$ is both distance and time preserving.

In this paper we focus on the geometry of timed-metric-spaces and the intrinsic timed-Hausdorff distance.   Our work applies to study Lorentzian manifolds and their intrinsic timed-Hausdorff limits, but we will not be exploring this within the paper.  Hence, we will not refer further to smooth manifolds.

\subsection{\bf Review of Timed-Fr\'echet Maps}\label{sect:SS2}
Sakovich--Sormani introduced the following notion in
\cite[Definition~4.21]{SakSor-Notions}:

\begin{defn}\label{defn:tau-K}
Given a compact timed-metric-space, $(X,d,\tau)$,
and a countably dense collection of points,
\be
{\mathcal N}=\{
x_1,x_2,x_3,\ldots\}\subset X,
\ee
we define a
timed-Fr\'echet map\footnote{In v1 of \cite{SakSor-Notions} these were accidentally named timed-Fr\'echet maps and this will be fixed.} 
\index{pa@$\kappa_{\tau,X}=\kappa_{\tau,X,{\mathcal N}}$}
$
\kappa_{\tau,X}
=\kappa_{\tau,X,{\mathcal N}}\colon X \to [0,\tau_{max}]\times \ell^\infty\subset \ell^\infty
$,
by
\be
\kappa_{\tau,X}(x)=
(\tau(x),\kappa_{X,{\mathcal N}}(x))=
(\tau(x), d(x_1,x),
d(x_2,x),
d(x_3,x),...).
\ee
\end{defn}

For the reader's convenience, we state here propositions~4.24 and 4.25 from \cite{SakSor-Notions}.

\begin{prop}\label{prop:tau-K-dist-pres}
Any
timed-Fr\'echet map, 
$
\kappa_{\tau,X}\colon X \to [0,\tau_{max}]\times \ell^\infty\subset \ell^\infty,
$
is distance preserving,
\be
d_{\ell^\infty}(\kappa_{\tau,X}(x),
\kappa_{\tau,X}(y))=d(x,y),
\ee
for any separable, bounded 
metric space, $(X,d)$, with a $1$-Lipschitz time function, $\tau\colon X\to[0,\tau_{max}]$
and any choice of countably dense points in $X$.
\end{prop}

\begin{prop}\label{prop:tau-K-time-pres}
If we consider the infinite dimensional
timed-metric-space
\be
([0,\tau_{max}]\times \ell^\infty,
d_{\ell^\infty},\tau_\infty
),
\ee
where 
\be
\tau_\infty(w_0,w_1,w_2,\ldots)=w_0,
\ee
and endowed with a causal structure
defined by
\be
\forall z \in [0,\tau_{max}]\times \ell^\infty \quad
w\in J_\infty^+(z)
\iff 
\tau_\infty(w)-\tau_\infty(z)=
d_{\ell^\infty}(w,z),
\ee
then every timed-Fr\'echet map 
$
\kappa_{\tau,X}\colon X \to [0,\tau_{max}]\times \ell^\infty\subset \ell^\infty
$
is time preserving:
\be\label{eq:tau-K-time-pres}
\tau_\infty(\kappa_{\tau,X}(x))=\tau(x)
\ee
and preserves causality:
\be\label{eq:tau-K-causal-pres}
\kappa_{\tau,X}(p)\in J_\infty^+(\kappa_{\tau,X}(q))
\iff p\in J_X^+(q)
\ee
for any separable, bounded metric space, $(X,d)$, with a $1$-Lipschitz time function, $\tau\colon X\to[0,\tau_{max}]$ that
encodes causality as in (\ref{eq:causal}).
\end{prop}

\subsection{\bf Review of the Intrinsic Timed-Hausdorff Distance}\label{sect:SS3}

In \cite[Definition~4.22]{SakSor-Notions} Sakovich and Sormani introduced the following notion within their definition of the 
intrinsic timed-Hausdorff
distance between causally null compactifiable space-times:

\begin{defn}\label{defn:timed-H}
The {\bf
intrinsic 
timed-Hausdorff
distance
between two compact timed-metric-spaces},
$(X_i, d_i, \tau_i)$, $i=1,2$, is given by
\be \label{eq:tK-GH-1}
d^\kappa_{\tau-H}
\Big((X_1,d_1,\tau_1),(X_2,d_2,\tau_2)\Big)=
\inf\, d^{\ell^\infty}_H(
\kappa_{\tau_1,X_1}(X_1),
\kappa_{\tau_2,X_2}(X_2))
\ee
where the infimum is taken over all possible 
timed-Fr\'echet maps,
$\kappa_{\tau_1,X_1}\colon X_1\to \ell^\infty$ and
$\kappa_{\tau_2,X_2}\colon X_2\to \ell^\infty$ which are
found by considering all countable dense collections of points in the $X_i$ and all reorderings of these collections of points.
\end{defn}

They proved this notion of distance is definite in \cite[Theorem 5.14]{SakSor-Notions}, which we state below.

\begin{thm}\label{thm:definite-timed-H}
The intrinsic timed-Hausdorff distance between two compact timed-metric-spaces, as in definition~\ref{defn:timed-H}, is definite in the sense that:
\be \label{eq:tau-K-definite-1}
d^\kappa_{\tau-H}\Big((X,d_X,\tau_X),(Y,d_Y,\tau_Y)\Big)=0
\ee
if and only if there is a 
distance and time preserving bijection
\be\label{eq:tau-K-definite-2}
F\colon (X,d_X,\tau_X)\to (Y,d_Y,\tau_Y).
\ee
\end{thm}

\section{\bf Proving the Compactness Theorem with Addresses}

In this section we prove our compactness theorem:

\begin{thm}
[Timed Gromov Compactness Theorem]\label{thm:timed-Gromov-compactness-A}
If $(X_j,d_j,\tau_j)$ is a sequence of compact timed-metric-spaces that are equibounded, 
\be\label{eq:equibounded-A}
\exists D>0 \,s.t.\, \diam_{d_j}(X_j)\le D,
\ee
equicompact,
\be
\forall R\in (0,D]\,
\exists N(R) \in \mathbb N \,s.t.\,
\exists \{x^j_{R,i}:i=1,\ldots,N(R)\}\,\, s.t.\,\,
X_j\subset \bigcup_{i=1}^{N(R)} B(x^j_{R,i},R)
\ee
and have a uniform bounds on their $1$-Lipschitz functions 
\be
\tau_j(X_j)\subset [0,\tau_{max}],
\ee
then a subsequence converges in the intrinsic timed-Hausdorff sense to a compact timed-metric-space $(\bar{X}_\infty, d_\infty, \tau_\infty)$. In fact, 
up to a subsequence, 
there exist distance and time preserving 
timed-Fr\'echet maps,
\be \label{eq:tK-A1}
\varphi_j\colon  X_j \to [0,\tau_{max}]\times Z \subset \ell^\infty,
\ee
where $Z$ is compact, such that 
\be 
d_H^{[0,\tau_{max}]\times Z}(\varphi_j(X_j),\varphi_\infty(X_\infty))\to 0.
\ee
In addition, there is an uncountable set $\mathcal{A}$,
and surjective maps,
\be \label{eq:I-A}
{\mathcal I}^j\colon {\mathcal A} \to X_j
\quad\textrm{ and }\quad
{\mathcal I}^\infty\colon {\mathcal A} \to \bar{X}_\infty
\ee
such that we have uniform convergence
as $j\to \infty$,
\be \label{eq:alpha-A}
\sup_{\alpha \in {\mathcal A}}
d_{\ell^\infty}
\Big(\varphi_j({\mathcal I}^j(\alpha)),
\varphi_\infty({\mathcal I}^\infty(\alpha)) \Big)
\to 0.
\ee
In particular, we have uniform convergence of time,
\be \label{eq:sup-A}
\sup_{\alpha \in {\mathcal A}}
|\tau_j({\mathcal I}^j(\alpha))-
\tau_\infty({\mathcal I}^\infty(\alpha))|\to 0,
\ee
and uniform convergence of distances,
\be \label{eq:sup-A-d}
\sup_{\alpha,\alpha' \in {\mathcal A}}
|d_j({\mathcal I}^j(\alpha),{\mathcal I}^j(\alpha'))-
d_\infty({\mathcal I}^\infty(\alpha),{\mathcal I}^\infty(\alpha'))|\to 0.
\ee
\end{thm}

\subsection{\bf Selection of Balls in a Compact Metric Space}\label{sect:p1}

The following proposition is from the first part of Gromov's proof of his Compactness Criterion in \cite{Gromov-poly}. Since we will apply it later to prove our Theorem~\ref{thm:timed-Gromov-compactness-A} we provide more details in the proof and add 
in the precise indexing of the points $x_a$. 
The proposition allows us to carefully select a countable dense collection of points covering each metric space and indexing them in a way which keeps track of where each point is located. See Figure~\ref{fig:index}. We will apply this indexing later to define the index set of ``addresses" $\mathcal{A}$.

\begin{prop}\label{prop:selection}
Suppose $(X_j,d_j)_j$ is a sequence of metric spaces
satisfying the equiboundedness and equicompactness conditions as in the statement
of Theorem~\ref{thm:Gromov-compactness}.  
Take the sequence $\varepsilon_i = 2^{-i}$ and let $N_i=N(\varepsilon_i)$ be natural numbers such that $X_j$
can be covered by $N_i$ balls of radius $\varepsilon_i$.  For each $i\in {\mathbb N}$, let $A_i$ be the finite set, 
\be \label{eq:index}
A_i= \Bigl\{(a_1, a_2, \ldots, a_i): \, 1 \leq a_j \leq N_j, j=1,\ldots, i \Bigr\},
\ee
$p_i: A_{i+1} \to A_i$ the projection maps
$p_i(a_1,\ldots,a_{i+1})=(a_1,\ldots,a_{i})$
and
\be\label{eq:A-cntble}
A=\bigsqcup_{i=1}^{\infty }A_i.
\ee
Then there exists a countable dense subset 
\be\label{eq:Nj}
{\mathcal{N}}_j = \{x_a^j:\, a\in A\} \subset X_j
\ee
and maps
\be\label{eq:Iji}
I^j_i \colon  A_i \to X_j \quad\textrm{ and }\quad I^j\colon  A\to X_j,
\ee
given by
\be
I^j(a)=I^j_i(a)=x_a^j \in X_j \qquad \forall a \in A_i \subset A, 
\ee
that satisfy the following conditions:
\newline
Firstly, each $I^j_i(A_i)$ is an $\varepsilon_i$-net of $X_j$, which means
\be \label{eq:cover}
X_j \subset \bigcup_{a\in A_i} B_{d_j}(x^j_a,\varepsilon_i).
\ee
Secondly, for each $i,j \in \mathbb N$ and $a \in A_{i+1}$,  
\be \label{eq:in-0}
I^j_{i+1}(a) \in B_{d_j}(I^j_i(p_i(a)), 2\varepsilon_i)\subset X_j,
\ee
which means
\be \label{eq:in}
x^j_{(a_1,...a_i,a_{i+1})}
\in B_{d_j}(x^j_{(a_1,...a_i)},2\varepsilon_i)
\ee
for any $a=(a_1,\dots,a_{i+1})\in A_{i+1}$.
\end{prop}

\begin{figure}[h] 
   \centering
   \includegraphics[width=5in]{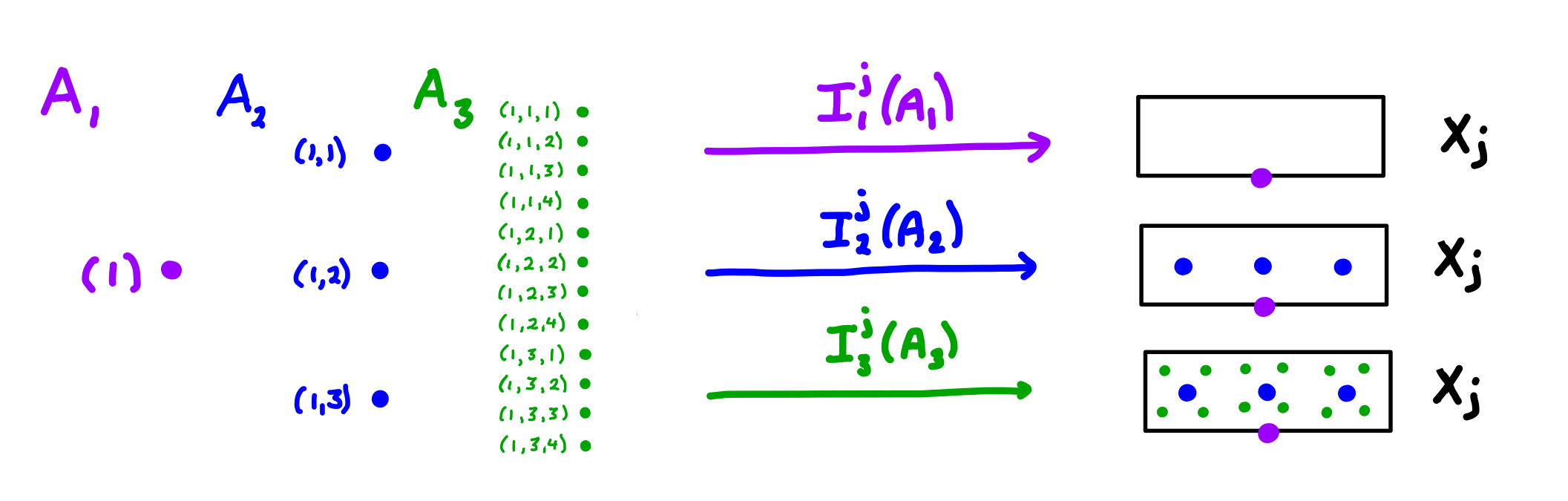}
   \caption{An index set with $N(1)=1, N(2)=3, N(3)=4$ and the corresponding points in $X_j$.
   }\label{fig:index}
\end{figure}

\begin{proof}
We will define $A_i$, $p_i$ and $I^j_i$ 
satisfying (\ref{eq:cover}) and (\ref{eq:in})
by induction on $i\in {\mathbb N}$. Note that, once this induction is completed, 
we can see that the 
set $I^j(A)$ is dense in $X_j$ because
we can take any $x\in X_j$ and any $\delta>0$, and choose
$i$ sufficiently large that $\varepsilon_i<\delta$, and
then apply (\ref{eq:cover}) to see that, 
\be
\exists x^j_a\in I^j(A)
\,\,s.t.\,\,d_j(x, x^j_a)<\varepsilon_i<\delta.
\ee
By construction, the set $A$ and its image $I^j(A)$ are
countable because they are countable unions of finite sets, $A_i$, and their images, $I^j_i(A_i)$, respectively. Thus 
we need only complete the induction to prove this proposition.

We begin the induction argument with $i=1$.  By equiboundedness and the definition of
$N_1=N(\varepsilon_1)$, we have
an $\varepsilon_1$-net of points,
\be
\{x^j_{(1)},x^j_{(2)},x^j_{(3)},...x^j_{(N_1)}\}\subset X 
\ee
indexed by 
\be
A_1=\{(1),(2),\ldots,(N_1)\}
\ee
such that 
\be \label{eq:BC1}
X_j \subset \bigcup_{a\in A_1} B_{d_j}(x^j_a,\varepsilon_1)
=\bigcup_{a_1=1}^{N_1} B_{d_j}(x^j_{(a_1)},\varepsilon_1)
\ee
which gives us (\ref{eq:cover}) for
$i=1$.   

By equiboundedness and the definition of
$N_2=N(\varepsilon_2)$, for each ball of radius $\varepsilon_1$ (indexed by $a_1=1,2,...,N_1$) as above, 
we can choose $N_2$ points in $X_j$ such that
\be \label{eq:BC2}
B_{d_j}(x^j_{(a_1)},\varepsilon_1) \subset
\bigcup_{a_2=1}^{N_2} B_{d_j}(x^j_{(a_1,a_2)},\varepsilon_2)
\ee
with 
\be
B_{d_j}(x^j_{(a_1)},\varepsilon_1) \cap
B_{d_j}(x^j_{(a_1,a_2)},\varepsilon_2)\neq \emptyset  \qquad a_2=1,\ldots,N_2. 
\ee
Thus 
\be\label{eq:in-1}
x^j_{(a_1,a_2)} \in  B_{d_j}(x^j_{(a_1)},2\varepsilon_1)\subset X_j,
\ee
So we have an index set,
\be
A_2=\{(a_1,a_2):  1 \leq a_1\leq N_1,\, 1 \leq a_2 \leq N_2\},
\ee
and a projection map
\be
p_1 \colon  A_2 \to A_1 \qquad p_1\colon (a_1,a_2) \mapsto (a_1)
\ee
and the map
$
I^j_2 \colon  A_2 \to X_j
$
identifying the center points, 
\be
x^j_{(a_1,a_2)}=I^j_2((a_1,a_2))
\ee
of the balls
of radius $\varepsilon_2$.
Observe that by (\ref{eq:in-1})
and the definition of the 
projection map, $p_1$, 
we have indexed the points
\be\label{eq:in-1x}
x^j_{(a_1,a_2)} \in  B_{d_j}(x^j_{p_1(a_1,a_2)},2\varepsilon_1)\subset X_j,
\ee
which gives
(\ref{eq:in}) for $i=1$.
This completes the base step of our
proof by induction.

For further understanding, note that by combining 
(\ref{eq:BC1}) and 
(\ref{eq:BC2}) we have
\[
X_j \subset
\bigcup_{a_1=1}^{N_1} B_{d_j}(x^j_{(a_1)},\varepsilon_1)
\subset
\bigcup_{a_1=1}^{N_1}
\bigcup_{a_2=1}^{N_2} B_{d_j}(x^j_{(a_1,a_2)},\varepsilon_2)
=
\bigcup_{a\in A_2} B_{d_j}(x^j_a,\varepsilon_2) 
\]
which gives us (\ref{eq:cover}) for
$i=2$.

Now we proceed with the induction step.
Assume we already have a collection
of points indexed by the
sets 
\begin{eqnarray*}
A_1&=&\{(a_1):
1\leq a_1\leq N_1\},\\
A_2&=&\{(a_1,a_2):  
1\leq a_1\leq N_1,\, 
1\leq a_2 \leq N_2\},\\
A_3&=&\{(a_1,a_2,a_3): 
1 \leq a_1 \leq N_1, \,
1 \leq a_2 \leq N_2,\, 
1 \leq a_3 \leq N_3\},\\
&\vdots&\\
A_i&=& \{(a_1, a_2, \ldots, a_i): \,
1 \leq a_1 \leq N_1, \,
1 \leq a_2 \leq N_2,\ldots,
1 \leq a_i \leq N_i \},
\end{eqnarray*}
with projection maps, 
\be
p_1\colon A_2\to A_1, \quad
p_2\colon A_3\to A_2, \quad \ldots, \quad p_{i-1}\colon A_i\to A_{i-1},
\ee
and maps
\be
I^j_1 \colon A_1 \to X_j, \quad
I^j_2 \colon A_2 \to X_j, \quad \ldots,\quad
I^j_i \colon A_i \to X_j,
\ee
identifying the center points, 
\begin{eqnarray*}
x^j_{(a_1,a_2)} \,\,=
&I^j_2((a_1,a_2))&\in\,\,  
B_{d_j}(x^j_{(a_1)},2\varepsilon_1),\\
x^j_{(a_1,a_2,a_3)}\,\,
=
&I^j_3((a_1,a_2,a_3))& \in\,\,  
B_{d_j}(x^j_{(a_1,a_2)},2\varepsilon_2),\\
&\vdots&\\
x^j_{(a_1,a_2,...,a_i)}\,\,
=&I^j_i((a_1,a_2,...,a_i)) &
\in\,\,  
B_{d_j}(x^j_{(a_1,a_2,...,a_{i-1})},2\varepsilon_{i-1}),
\end{eqnarray*}
such that
\begin{eqnarray*}
B_{d_j}(x^j_{(a_1)},\varepsilon_1)
&\subset &
\bigcup_{a_2=1}^{N_2}
B_{d_j}(x^j_{(a_1,a_2)},\varepsilon_2),
\\
B_{d_j}(x^j_{(a_1,a_2)},\varepsilon_2)
&\subset &
\bigcup_{a_3=1}^{N_3}
B_{d_j}(x^j_{(a_1,a_2,a_3)},\varepsilon_3),\\
&\vdots&\\
B_{d_j}(x^j_{(a_1,a_2,...,a_{i-1})},\varepsilon_{i-1})
&\subset &
\bigcup_{a_i=1}^{N_i}
B_{d_j}(x^j_{(a_1,a_2,...,a_{i-1},a_i)},\varepsilon_i).
\end{eqnarray*}
By equicontinuity, each ball of radius $\varepsilon_i$
centered at 
\be
x^j_{(a_1,a_2,...,a_i)}=I^j_i((a_1,a_2,...,a_i)),
\ee
can be covered by $N_{i+1}$ balls of radius $\varepsilon_{i+1}$ centered at points
\be\label{eq:in+1}
x^j_{(a_1,a_2,...,a_i,a_{i+1})}\,\,
\in\,\,  
B_{d_j}(x^j_{(a_1,a_2,...,a_{i-1},a_i)},2\varepsilon_{i}),
\ee
where $a_{i+1}\in \{1,2,...,N_{i+1}\}$.
That is
\be\label{eq:cover+1}
B_{d_j}(x^j_{(a_1,a_2,...,a_{i})},\varepsilon_{i})
\,\,\subset\,\, 
\bigcup_{a_{i+1}=1}^{N_{i+1}}
B_{d_j}(x^j_{(a_1,a_2,...,a_i,a_{i+1})},\varepsilon_{i+1}).
\ee
So we have all these new points together indexed by
the new index set,
\be
A_{i+1}= \{(a_1, a_2, \ldots, a_{i+1}):
1 \leq a_1 \leq N_1, \,
1 \leq a_2 \leq N_2,\ldots,
1 \leq a_{i+1} \leq N_{i+1}\},
\ee
and new
maps $I^j_{i+1} \colon  A_{i+1} \to X_j$ defined by
\be
x^j_{(a_1,a_2,...,a_i,a_{i+1})}=
I^j_{i+1}((a_1,a_2,...,a_i,a_{i+1})).
\ee
The new projection map, 
\be
p_{i}\colon A_{i+1}\to A_{i},
\ee
with (\ref{eq:in+1}) gives (\ref{eq:in}) 
and with (\ref{eq:cover+1}) gives (\ref{eq:cover})
inductively.
\end{proof}

\subsection{\bf The Index Set of Addresses}
\label{sect:p2}

A natural consequence of Proposition~\ref{prop:selection}, is that we
have a uniform way of providing an ``addresses'' to all points in equicompact collections of metric spaces. More precisely, we have the following definition and proposition.

\begin{defn}
Given a set $A$ and maps $p_i$ as in Proposition~\ref{prop:selection},
an {\bf address} is a sequence
\be\label{eq:alpha}
\alpha=\{\alpha_1,\alpha_2,\alpha_3,...\} \subset A
\ee
such that 
\be\label{eq:alpha-a}
\alpha_i=(a_1,a_2,...a_i)\in A_i\subset A \qquad \forall i \in {\mathbb N},
\ee
satisfying
\be
p_i(\alpha_{i+1})=\alpha_i \qquad \forall i \in {\mathbb N}.
\ee
We denote by $\mathcal{A}$ the set of all such addresses (see Figure~\ref{fig:address}).  
\end{defn}

\begin{prop}\label{prop:addresses}
Let $q_i\colon {\mathcal A}\to A_i$ given by
$q_i(\alpha)=\alpha_i$. Then for any $X_j$ and maps $I^j_i\colon A_i\to X_j$ as in 
Proposition~\ref{prop:selection},
the map ${\mathcal I}^j\colon{\mathcal A}\to X_j$,
defined by
\be \label{eq:address-lim}
{\mathcal I}^j(\alpha)=\lim_{i\to \infty}I^j_i(q_i(\alpha))
\ee
is well defined and surjective.  That is, every
point $x\in X_j$, has an address, 
\be
\alpha=\alpha^j(x)=\{\alpha^j_1(x),\alpha^j_2(x),...\}
\in {\mathcal A},
\ee 
such that $x={\mathcal I}^j(\alpha^j(x))$.
In fact,
\be\label{eq:alpha-x}
x\in \bar{B}_{d_j}(I^j_i(\alpha^j_i(x)),\varepsilon_i)
\qquad \forall i\in {\mathbb N}.
\ee
We do not claim this address is unique. Moreover, we get the same result even if \eqref{eq:cover} and \eqref{eq:in-0} hold for closed balls rather than open balls.
\end{prop}

\begin{figure}[h] 
   \centering
   \includegraphics[width=5in]{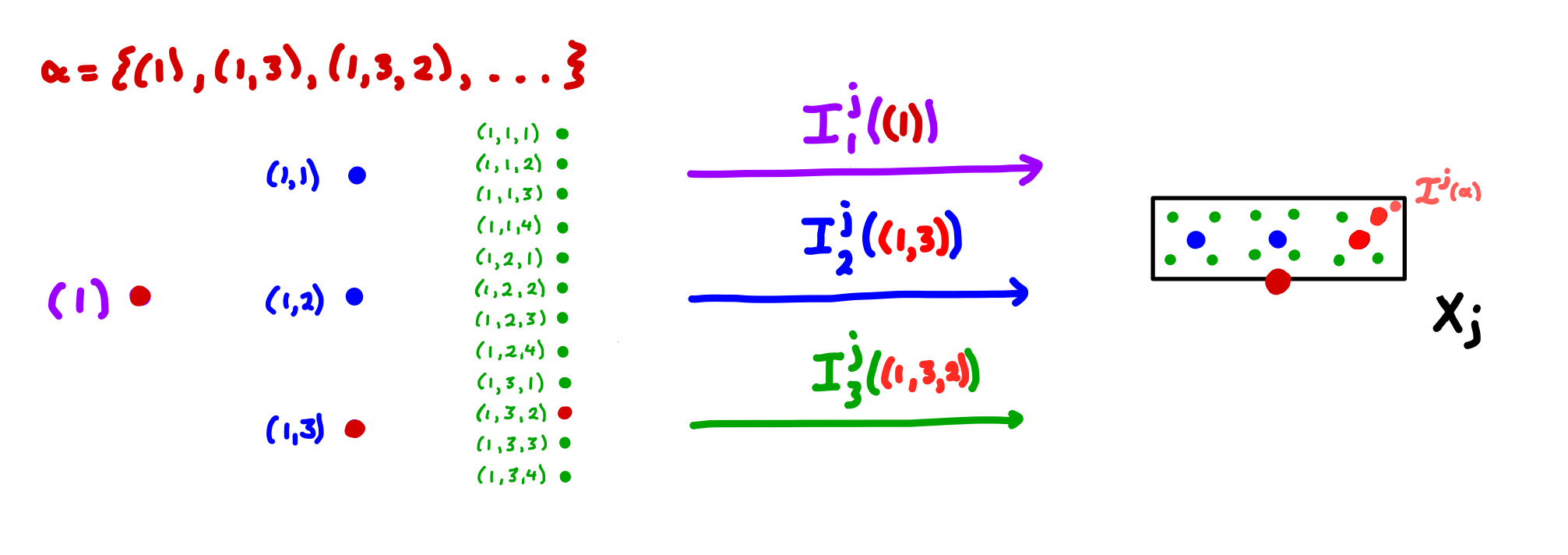}
   \caption{An address $\alpha$ identifying a sequence of points in $X_j$ approaching the point ${\mathcal I}^j(\alpha)$.
   }\label{fig:address}
\end{figure}

\begin{proof}
To prove ${\mathcal I}^j\colon{\mathcal A}\to X_j$ is well defined, we need only show the sequence of points,
\be
\{I^j_i(\alpha_i)\colon\, i\in \mathbb N\}
\ee
is Cauchy.   By $p_i(\alpha_{i+1})=\alpha_i$,
and the property \eqref{eq:in-0}
for $I^j_i$ proven in 
Proposition~\ref{prop:selection}, we have
\be
I^j_{i+1}(\alpha_{i+1})
\in 
\bar{B}_{d_j}(I^j_i(p_i(\alpha_{i+1})),2\varepsilon_i)=\bar{B}_{d_j}(I^j_i(\alpha_i),2\varepsilon_i).
\ee
Thus for all $i\in {\mathbb N}$, we have
\be
d_j(I^j_i(\alpha_i), I^j_{i+1}(\alpha_{i+1}))\le 2\varepsilon_i=1/2^{i-1}. 
\ee
Applying the triangle inequality and then this
estimate repeatedly we have
\begin{eqnarray*}
d_j(I^j_i(\alpha_i), I^j_{i+k}(\alpha_{i+k}))
&\le&
d_j(I^j_i(\alpha_i), I^j_{i+1}(\alpha_{i+1}))+\cdots
+d_j(I^j_{i+k-1}(\alpha_{i+k-1}), I^j_{i+k}(\alpha_{i+k}))\\
&\le&\tfrac{1}{2^{i-1}}+\cdots + \tfrac{1}{2^{i+k-2}}<\tfrac{1}{2^i}=\varepsilon_i.
\end{eqnarray*}
So the sequence is Cauchy.   Since $X_j$ is compact,
it is complete and so the limit in (\ref{eq:address-lim}) exists, defining ${\mathcal I}^j(\alpha)$.

Next we show surjectivity. Take any $x\in X_j$.
We then apply property \eqref{eq:cover}. Since
$X_j$ is covered by balls of radius $\varepsilon_1$ we know there exists $(a^j_1(x))\in A_1$ such that
\be
x\in \bar{B}_{d_j}(I^j_1(a^j_1(x)),\varepsilon_1).
\ee
This ball is covered by balls of radius $\varepsilon_2$,
with centers indexed by $(a_1,a_2)\in A_2$ so
there exists $a^j_2(x)\in \{1, \ldots, N_2\}$ such that
\be
x\in \bar{B}_{d_j}(I^j_2((a^j_1(x),a^j_2(x))),\varepsilon_2).
\ee
Continuing forward we select each $a^j_i(x)$ such that
\be
x\in \bar{B}_{d_j}(I^j_i((a^j_1(x),a^j_2(x),...,a^j_i(x))),\varepsilon_i).
\ee
Taking 
\be
\alpha^j(x)=\{\alpha^j_1(x),\alpha^j_2(x),...\}
\in \mathcal{A}
\ee
where
\be
\alpha^j_i(x)=(a^j_1(x),a^j_2(x),...,a^j_i(x))
\ee
gives us (\ref{eq:alpha-x}).  By the
definition of ${\mathcal I}^j$, and then
by (\ref{eq:alpha-x}), we have
\be
{\mathcal I}^j(\alpha^j(x))
=\lim_{i\to \infty} I^j_i(\alpha^j_i(x))=x.
\ee
\end{proof}

\subsection{\bf Compactness Theorem with Addresses}
\label{sect:p3}

We can now prove our compactness theorem.

\begin{proof}[Proof of Theorem~\ref{thm:timed-Gromov-compactness-A}]
The proof begins as in Gromov's Compactness Criterion in \cite{Gromov-poly} using the
equicompactness as we have stated in Proposition~\ref{prop:selection}
to define a
countable index set $A$
defined as the disjoint union of finite sets,
\be
A=\bigsqcup_{i=1}^{\infty} A_i,
\ee
and a countable dense collection of points
\be
I^j(A)=\{x^j_{a}=I^j(a):\,a\in A\}\subset X_j
\ee
such that each $I^j(A_i)$ is a finite
$\varepsilon_i=1/2^i$ net for $X_j$
and has all the other properties described in Proposition~\ref{prop:selection}.

Then we follow the proof in \cite{BBI} of Gromov's Compactness Theorem, to construct a limit space. 
By the equiboundedness of the distance, for any $a,b \in A$, we define 
\be\label{eq:djab}
d^j_{a,b}=d_j(x^j_{a},x^j_{b})\in [0,D].
\ee
This definition is
symmetric:
\be
d^j_{a,b}=d^j_{b,a}\qquad \forall a,b\in A;
\ee
and it satisfies the triangle inequality:
\be
d^j_{a,c}\le d^j_{a,b}+d^j_{b,c}
\qquad \forall a,b,c\in A.
\ee
By the equiboundedness of the time functions,
for any $a\in A$, we define 
\be\label{eq:tja}
\tau^j_{a}=\tau_j(x^j_{a})\in [0,\tau_{max}].
\ee
Since the time functions are $1$-Lipschitz, we have 
\be\label{eq:tjLip}
|\tau^j_{a}-\tau^j_{b}|\le d^j_{a,b}
\qquad \forall a,b\in A.
\ee

By the countability of the index set $A$, we take subsequences of subsequences and diagonalize so that for the final subsequence (which we still index with $j$), we have for each
$a\in A$, convergence of the real numbers:
\be \label{eq:limtja}
\lim_{j\to\infty}\tau^j_{a}=\tau^\infty_{a}\in [0,\tau_{max}].
\ee
We take further subsequences and diagonalize further, so that for the final subsequence (which we still index with $j$), we have for each
$a,b\in A$,
\be \label{eq:limdjab}
\lim_{j\to\infty}d^j_{a,b}=d^\infty_{a,b}\in [0,D].
\ee
Taking the limit of (\ref{eq:tjLip})
we know
\be\label{eq:tlimLip}
|\tau^\infty_{a}-\tau^\infty_{b}|\le d^\infty_{a,b}
\qquad \forall a,b\in A.
\ee

We now define a pseudo-metric space $(A, d_\infty)$ with
\be
d_{\infty}(a,b)=d^\infty_{a,b}\in [0,D]
\ee
which satisfies the triangle inequality and is symmetric but
is not necessarily definite.
By (\ref{eq:tlimLip}) we know that
\be \label{eq:tdefwell}
d^\infty(a,b)=0\,\,\implies\,\,
\tau^\infty_{a}=\tau^\infty_{b}.
\ee
Thus we can mod out by points of zero
distance apart to define a timed
metric space which is definite:
\be
(X_\infty,d_\infty,\tau_\infty)=(\{[a]:a\in A \}, d_\infty,\tau_\infty)
\ee
where the points are equivalence classes,
\be
[a]=\{b\in A:\, d_\infty(a,b)=0\},
\ee
the distance function is
\be \label{eq:dinftyab}
d_\infty([a],[b])=d^\infty_{a,b} \in [0,D],
\ee
and the time function is
\be\label{eq:tinftya}
\tau_\infty([a])=\tau^\infty_{a} \in [0,\tau_{\max}],
\ee
which is well defined by
(\ref{eq:tdefwell}) and is
$1$-Lipschitz by (\ref{eq:tlimLip}).

We next take the metric completion
\be
(\bar{X}_\infty,d_\infty,\tau_\infty)
\ee
which has the usual definite
distance, $d_\infty$, extended 
from $X_\infty$ to $\bar{X}_\infty$
satisfying symmetry and
the triangle inequality, with
\be
\diam_{d_\infty}(\bar{X}_\infty)\le D,
\ee
and a bounded time function, $\tau_\infty$, extended as a $1$-Lipschitz function 
from $X_\infty$ to $\bar{X}_\infty$.

We claim that $\bar X_\infty$ is compact. Note that to prove this claim is enough to show that for each $i \in \mathbb N$,
\be\label{eq-eps_iNet}
\Bigl \{ [a] \, : \, a \in A_i \Bigr \} \qquad \text{is an $\varepsilon_i$-net of $X_\infty$.}
\ee 
Take $[b] \in X_\infty$. Thus, there exists $k \in \mathbb N$ such that $b \in A_k$. Since for each $j \in \mathbb N$ we know that $I_i^j(A_i)$ is an $\varepsilon_i$-net of $X_j$, by the pigeon-hole principle, there exists $a \in A_i$ such that 
\be
d_j(I_i^j(a),I_i^k(b)) \leq \varepsilon_i
\ee 
for infinitely many indices $i$. 
Thus, up to a  subsequence, 
\be
d_\infty([b],[a])= \lim_{j \to \infty}d_j(x_a^j, x_b^j) \leq \varepsilon_i.
\ee
Hence, we have \eqref{eq-eps_iNet}. 

Up to here we have followed the proof from \cite{BBI} of Gromov's Compactness Theorem, 
with corresponding additions to address the fact that we have time functions, 
to construct a compact timed-metric space $(\bar X_\infty, d_\infty, \tau_\infty)$. 
Since we will show timed-Hausdorff convergence, 
\be
(X_j,d_j,\tau_j) \tHto (\bar{X}_\infty,d_\infty,\tau_\infty),
\ee
using timed-Fr\'echet maps as in Definition \ref{defn:timed-H}, from now on, we proceed in a different manner to \cite{BBI}. 

Note that we can define timed-Fr\'echet maps,
\be \label{tK1}
\kappa_{\tau_j,{\mathcal N}_j}\colon 
X_j\to [0,\tau_{max}]\times
\prod_{a\in A} [0,D] \subset \ell^\infty
\ee
by
\be\label{tK2}
\kappa_{\tau_j,{\mathcal N}_j}(x)=(\tau_j(x),\kappa_{{\mathcal N}_j}(x))
\ee
where 
\be\label{tK3}
\kappa_{{\mathcal N}_j}(x)\in \prod_{a\in A} [0,D]
\ee
has components 
\be\label{tK4}
(\kappa_{\mathcal N_j}(x))_a=d_j(I^j(a),x).
\ee

To define a timed-Fr\'echet map for 
$(\bar X_\infty, d_\infty, \tau_\infty)$, 
recall that we have a countable index set, $A$,
which is the countable disjoint union of finite index sets, $A_i$.
We define
\be
I^\infty_i \colon A_i \to X_\infty\subset \bar{X}_\infty\quad \textrm{ and }\quad I^\infty\colon A\to X_\infty\subset \bar{X}_\infty,
\ee
by 
\be
I^\infty(a)=I^\infty_i(a)=[a]\in X_\infty\subset \bar{X}_\infty \qquad \forall \, a \in A_i \subset A,
\ee 
and we denote these points as
\be\label{eq:points2}
x^\infty_a=I^\infty(a)=I^\infty_i(a).
\ee
Thus,
\be
{\mathcal{N}}_\infty=I^\infty(A)=\{x_a^\infty:\, a\in A\}
\ee
is a countable
ordered collection of points, that is dense in $\bar X_\infty$, since as a set it equals $X_\infty$.
From this, and recalling \eqref{eq:dinftyab} and \eqref{eq:tinftya},
 we also have a timed-Fr\'echet map $\kappa_{\tau_\infty, \mathcal{N_\infty}}
 $ satisfying
 (\ref{tK1})-(\ref{tK4})
for $j=\infty$.

Before proving \eqref{eq:H-A} we first 
show that the uniform estimate \eqref{eq:alpha-A} holds.
As a preliminary step we note that we have uniform convergence on finite unions of the $A_i$:
Recall that by (\ref{eq:djab}), (\ref{eq:limdjab})
and (\ref{eq:dinftyab}), 
for each
$a,b\in A$, we have pointwise convergence, 
\be\label{eq:limdabcntble}
d_\infty(x^\infty_a,x^\infty_b)=
\lim_{j\to\infty}d^j(x^j_a,x^j_b) \in [0,D].
\ee
By (\ref{eq:tja}), (\ref{eq:limtja})
and (\ref{eq:tinftya})
for each $a\in A$, we have pointwise convergence,
\be\label{eq:limtacntble}
\tau_\infty(x^\infty_a)=\lim_{j\to\infty}\tau_j(x^j_a).
\ee

Now, since each $A_i$ is finite, we do obtain uniform convergence on finite unions of the $A_i$: 
\be\label{eq:exists-Neps}
\forall \epsilon>0, \exists N_{\epsilon}^i\in {\mathbb N},
\ee
such that for all $j\ge N_\epsilon^i$ we have
\be\label{eq:unif-max}
\max\{|\tau_\infty(x^\infty_a)-\tau_j(x^j_a)|,
|d_\infty(x^\infty_a,x^\infty_b)-d^j(x^j_a,x^j_b)|\}<\epsilon
\ee
where the maximum is over
\be
a,b \in A_1\cup ... \cup A_i.
\ee

To get better control than just the pointwise convergence as in 
\eqref{eq:limdabcntble} and \eqref{eq:limtacntble}, 
we will use Proposition~\ref{prop:addresses} to provide an
address, $\alpha$, for every point in each $X_j$
and $\bar X_\infty$.

Note that we have proven that $\bar{X}_\infty$ satisfies 
the conclusions (\ref{eq:Iji})-(\ref{eq:Nj}), 
of Proposition~\ref{prop:selection}. It satisfies \ref{eq:cover}
by \eqref{eq-eps_iNet}.  Moreover, it easily satisfies (\ref{eq:in-0}) as follows.  
Let $a \in A_{i+1}$,  
\be 
I^\infty_{i+1}(a) \in B_{d_\infty}(I^\infty_i(p_i(a)), 2\varepsilon_i)\subset \bar X_\infty
\ee
 We know that for each $j \in \mathbb N$, 
\be
d_j(x^j_a, x^j_{p_i(a)}) \leq 2\varepsilon_i.
\ee
Hence, in the limit, we have 
\be
d_\infty(x^\infty_a, x^\infty_{p_i(a)}) = \lim_{j \to \infty} d_j(x^j_a, x^j_{p_i(a)}) \leq 2\varepsilon_i,
\ee
which is the desired estimate.

Since we have applied
Proposition~\ref{prop:selection} for each $X_j$, $j \in \mathbb N$, and the previous paragraph, 
for each $j \in \mathbb N \disjointunion \{\infty\}$
there exist surjective maps, ${\mathcal I}^j\colon{\mathcal A}\to X_j$, $j \in \mathbb N \cup \{\infty\}$
defined by
\be
{\mathcal I}^j(\alpha)=\lim_{i\to \infty}I^j_i(q_i(\alpha))
\ee
that satisfy the conditions of Proposition \ref{prop:addresses}.

We claim that for any $\varepsilon>0$  there exist $N_\varepsilon \in \mathbb N$ such that for all $j \geq N_\varepsilon$ we have for any address $\alpha$,
\be
|\tau_j({\mathcal I}^j(\alpha))-\tau_\infty({\mathcal I}^\infty(\alpha))| \leq \varepsilon
\ee
and that additionally for any $a\in A$, we also have
\be
|\kappa_{j,a}({\mathcal I}^j(\alpha))-\kappa_{\infty,a}({\mathcal I}^\infty(\alpha))| \leq \varepsilon.
 \ee

Let $\delta= \varepsilon/3$ and $i \in \mathbb N$ such that $\varepsilon_i < \delta$. Then 
there exists $N_{\delta}^i\in {\mathbb N}$, as in \eqref{eq:exists-Neps}, 
such that for all $j\ge N_\delta^i$ we have
(\ref{eq:unif-max}). 

We obtain the first inequality
using the triangle inequality, 
\begin{align}
|\tau_j({\mathcal I}^j(\alpha))-\tau_\infty({\mathcal I}^\infty(\alpha))| 
\,\,
\leq \,\, & |\tau_j({\mathcal I}^j(\alpha))-
\tau_j(I^j_i(\alpha_i))| \\ 
& +|\tau_j(I^j_i(\alpha_i))-\tau_\infty(I^j_i(\alpha_i))| \\
& + |\tau_\infty({\mathcal I}^\infty(\alpha))-
\tau_\infty(I^\infty_i(\alpha_i))| \\
\,\,\le \,\, & \,\,
d_j({\mathcal I}^j(\alpha),I^j_i(\alpha_i))  \\
 & +|\tau_j(I^j_i(\alpha_i))-\tau_\infty(I^j_i(\alpha_i))| \\
&+ 
d_\infty({\mathcal I}^\infty(\alpha),I^\infty_i(\alpha_i)) \\
\le \,\,& \,\,\varepsilon_i + \delta + \varepsilon_i \,\,
\leq \,\,\varepsilon \qquad \forall j \ge N^i_\delta, \quad \quad \quad
\end{align}
where for the second and third term in the right hand side we used that $\tau_j, \tau_\infty$ are $1$-Lipschitz and (\ref{eq:alpha-x}), and for the term in the middle we use  (\ref{eq:unif-max}). 
Similarly, we get the second inequality, 
\begin{align}
|\kappa_{j,a}({\mathcal I}^j(\alpha))-\kappa_{\infty,a}({\mathcal I}^\infty(\alpha))| 
\,\,\leq \,\,&
|\kappa_{j,a}({\mathcal I}^j(\alpha))-
\kappa_{j,a}(I^j_i(\alpha_i))| \\
&+  \,\, |\kappa_{j,a}(I^j_i(\alpha_i))
-\kappa_{\infty,a}(I^\infty_i(\alpha_i))| \\
&+ \,\,|\kappa_{\infty,a}({\mathcal I}^\infty(\alpha))-
\kappa_{\infty,a}(I^\infty_i(\alpha_i))| \\
\,\,\le \,\,&
d_j({\mathcal I}^j(\alpha),I^j_i(\alpha_i)) \\
 & + \,\,|\kappa_{j,a}(I^j_i(\alpha_i))
-\kappa_{\infty,a}(I^\infty_i(\alpha_i))|\\
&+ \,\, d_\infty({\mathcal I}^\infty(\alpha),I^\infty_i(\alpha_i)) \\
\,\,\leq \,\,& \,\,\varepsilon_i + \delta + \varepsilon_i \leq \varepsilon \qquad \forall j \ge N^i_\delta,\quad \quad \quad
\end{align}
where we used that $\kappa_j, \kappa_\infty$ are $1$-Lipschitz, (\ref{eq:alpha-x}) and  (\ref{eq:unif-max}).

In particular,
we have uniform convergence
\be
\tau_j\circ{\mathcal I}^j \to \tau_\infty\circ{\mathcal I}^\infty \textrm{ on }{\mathcal{A}}
\ee
and
\be
\kappa_{j,a}\circ{\mathcal I}^j \to \kappa_{\infty,a}\circ{\mathcal I}^\infty \textrm{ on }{\mathcal{A}}.
\ee
Hence, \eqref{eq:sup-A} and \eqref{eq:alpha-A} hold, where $\varphi_j=(\tau_j, \kappa_j)$ and   
$\varphi_\infty=(\tau_\infty, \kappa_\infty)$.

Now the correspondences 
\be
\mathcal{C}_j = \{(\kappa_{\tau_j,\mathcal{N}_j}({\mathcal I}^j(\alpha)),\kappa_{\tau_\infty,\mathcal{N}_\infty}({\mathcal I}^\infty(\alpha))):\alpha\in \mathcal{A}\}
\ee
between $\kappa_{\tau_j,\mathcal{N}_j}(X_j)$ and $\kappa_{\tau_\infty,\mathcal{N}_\infty}(\bar{X}_\infty)$, for all $j\in \mathbb{N}$, and the uniform convergence \eqref{eq:alpha-A} imply 
that 
\be\label{eq-Hconvlinf}
d^{\ell^\infty}_H(\kappa_{\tau_j,\mathcal{N}_j}(X_j),\kappa_{\tau_\infty,\mathcal{N}_\infty}(\bar{X}_\infty))\leq \|\varphi_j({\mathcal I}^j(\alpha))-
\varphi_\infty({\mathcal I}^\infty(\alpha))\|_\infty \to 0.
\ee

Finally, let us define 
\be
Z = \kappa_{\mathcal{N}_\infty}(\bar{X}_\infty)\cup\bigcup_{j\in \mathbb{N}} \kappa_{\mathcal{N}_j}(X_j) \subset \prod_{a\in A} [0,D] \subset \ell^\infty,
\ee
so $\kappa_{\tau_j,\mathcal{N}_j}\colon X_j\to [0,\tau_{max}]\times Z$ for all $j \in \mathbb{N}\cup\{\infty\}$. 

We claim that $Z$ is compact. 
Let $\{z_j\} \subset Z$. If there exists $J \in \mathbb N$ such that 
an infinite number of the $z_j$'s is contained in $\kappa_{\mathcal{N}_J}(X_J)$ then by compactness of $\kappa_{\mathcal{N}_J}(X_J)$
there is a further subsequence that converges to a point in $\kappa_{\mathcal{N}_J}(X_J) \subset Z$. Otherwise,  assume by passing to a subsequence that $z_j \in \kappa_{\mathcal{N}_j}(X_j)$ for all $j \in \mathbb N$.  By \eqref{eq-Hconvlinf}, for each $j \in \mathbb N$ there exists $w_j \in \kappa_{\mathcal{N}_\infty}(\bar X_\infty)$ such that 
\be
d_{\ell^\infty}(z_j, w_j)\to 0. 
\ee
Since we have shown that $\bar X_\infty$ is compact, there is a subsequence $\{w_{j_k}\}$ that converges to a point $w \in \kappa_{ \mathcal{N}_\infty}(\bar X_\infty)\subset Z$. It follows that $\{z_{j_k}\}$ converges to $w$. This completes the proof of $Z$ being compact, and finishes the proof of the theorem.
\end{proof}

\section{\bf Applications of Addresses}

In this final section we restate Gromov's Compactness Theorem and an Arzel\`a--Ascoli Theorem in terms of addresses and pose an open question about a more general version of Arzel\`a-Ascoli.

\subsection{\bf Gromov's Compactness Theorem with Addresses}\label{sect:ap1}

\begin{thm}
[Gromov Compactness Theorem with Addresses]\label{thm:Gromov-compactness-A}
If $(X_j,d_j)$ is a sequence of compact metric spaces that are equibounded
as in (\ref{eq:equibounded})
and
equicompact
as in 
(\ref{eq:equicompact})
then a subsequence converges in the Gromov--Hausdorff sense to a compact metric space $(X_\infty, d_\infty)$. In fact, there exist distance preserving 
Fr\'echet maps,
\be \label{eq:tK-A1-G}
\kappa_j\colon X_j \to  Z \subset \ell^\infty,
\ee
where $Z$ is compact, such that 
\be \label{eq:H-A-G}
d_H^{Z}(\kappa_j(X_j),\kappa_\infty(X_\infty))\to 0.
\ee
In addition, there is an uncountable set of addresses, $\mathcal{A}$,
and surjective index maps,
\be \label{eq:I-A-G}
{\mathcal I}^j\colon{\mathcal A} \to X_j
\quad \textrm{ and }\quad 
{\mathcal I}^\infty\colon{\mathcal A} \to \bar{X}_\infty
\ee
such that
\be \label{eq:alpha-A-G}
\sup_{\alpha \in {\mathcal A}}
d_{\ell^\infty}
\Big(\kappa_j({\mathcal I}^j(\alpha)),\kappa_\infty({\mathcal I}^\infty(\alpha)) \Big) \to 0
\ee
as $j\to \infty$. In particular,
\be \label{eq:sup-A-d-G}
\sup_{\alpha,\alpha' \in {\mathcal A}}
|d_j({\mathcal I}^j(\alpha),{\mathcal I}^j(\alpha'))-
d_\infty({\mathcal I}^\infty(\alpha),{\mathcal I}^\infty(\alpha'))|\to 0.
\ee
\end{thm}

This theorem follows exactly as in the proof of our Theorem~\ref{thm:timed-Gromov-compactness-A} simply by removing all references to the time function.   For completeness of exposition, we prove it as a trivial corollary of our theorem.

\begin{proof}
Take $\tau_j$ to be the constant zero valued function.  
Then all the hypotheses of Theorem~\ref{thm:timed-Gromov-compactness-A} are satisfied. Thus the conclusions hold. Finally apply
\be
\varphi_j(x)=
(\tau_j(x),\kappa_j(x))=(0,\kappa_j(x)).
\ee
\end{proof}

\subsection{\bf Arzel\`a--Ascoli Theorem with Addresses}
\label{sect:ap2}

\begin{thm}
[Arzel\`a--Ascoli-with-Addresses]\label{thm:ArzAsc-A}
Given a sequence of compact metric spaces:
\be
(X_j,d_j) \GHto (X_\infty,d_\infty),
\ee
and uniformly bounded $K$-Lipschitz functions for some $K>0$,
\be
F_j\colon X_j\to [0,F_{max}],
\ee
then a subsequence, also denoted $F_j$, 
converges to a 
bounded $K$-Lipschitz function
\be
F_\infty\colon X_\infty\to  [0,F_{max}]
\ee
in the following sense:
There exist distance preserving 
Fr\'echet maps,
\be
\kappa_j\colon X_j \to  Z \subset \ell^\infty,
\ee
where $Z$ is compact, such that 
\be
d_H^{Z}(\kappa_j(X_j),\kappa_\infty(X_\infty))\to 0,
\ee
\be \label{eq:sup-A-K}
\sup_{\alpha \in {\mathcal A}}
|F_j({\mathcal I}^j(\alpha))-
F_\infty({\mathcal I}^\infty(\alpha))|\to 0.
\ee
and
\be \label{eq:alpha-A-K}
\sup_{\alpha \in {\mathcal A}}
d_{\ell^\infty}
\Big(\kappa_j({\mathcal I}^j(\alpha)),
\kappa_\infty({\mathcal I}^\infty(\alpha)) \Big)\to 0,
\ee
where 
\be 
{\mathcal I}^j\colon{\mathcal A} \to X_j
\quad\textrm{ and }\quad
{\mathcal I}^\infty\colon{\mathcal A} \to X_\infty
\ee
are surjective maps and 
$\mathcal{A}$ is an uncountable set of addresses.  
\end{thm}

\begin{proof}
Gromov proved that with a compact GH limit, we have an
equibounded and equicompact sequence
of metric spaces \cite{Gromov-1981}.  
Setting
\be\label{eqFjtauj}
\tau_j= \tfrac{1}{K}F_j
\ee
we have timed-metric-spaces, $(X_j, d_j, \tau_j)$ satisfying the
hypotheses of Theorem~\ref{thm:timed-Gromov-compactness-A}.
The conclusions immediately follow from that theorem
and re-applying (\ref{eqFjtauj}).
\end{proof}

\subsection{\bf Arzel\`a--Ascoli Conjecture with Addresses}
\label{sect:ap3}

It would be interesting to apply the address method to more precisely describe the following stronger Arzel\`a--Ascoli Theorem:

\begin{conj}
Given a pair of Gromov--Hausdorff converging sequences of compact metric spaces:
\be
X_j \GHto X_\infty
\quad\textrm{ and }\quad Y_j \GHto Y_\infty
\ee
and given uniformly bounded Lipschitz $K$ functions for some $K>0$,
\be
F_j\colon X_j\to Y_j,
\ee
then a subsequence 
converges to a 
bounded Lipschitz $K$ function
\be
F_\infty\colon X_\infty\to Y_\infty.
\ee
\end{conj}

A proof of this theorem can be found in various places including \cite{Sormani-ArzAsc}.   However a new proof using the addresses would be able to describe the convergence of the functions more precisely as a kind of uniform convergence.

\section*{{\bf Appendix A: Triangle Inequality for the Intrinsic Timed-Hausdorff Distance} }\label{sec-appendix}

For the sake of completeness, we prove the following proposition, which, to the best of our knowledge, has not previously appeared in the literature.

\begin{prop}\label{prop:triangle inequality for d_tau-H}
Let $(X_i, d_i, \tau_i)$, $i=1,2,3$,  be three compact timed-metric-spaces. Then, it holds 
\begin{align}\label{eq:tK-H123}
& d^\kappa_{\tau-H}
\Big((X_1,d_1,\tau_1),(X_3,d_3,\tau_3)\Big)  \leq   \\
& d^\kappa_{\tau-H}
\Big((X_1,d_1,\tau_1),(X_2,d_2,\tau_2)\Big) +
d^\kappa_{\tau-H}
\Big((X_2,d_2,\tau_2),(X_3,d_3,\tau_3)\Big).
\end{align}
\end{prop}

The proof of the previous proposition proceeds along the same lines as that of the following result.

\begin{prop}\label{pop:triangle inequality for d_kappa-GH}
Let $X_1,X_2,X_3$ be compact metric spaces. Then 
\begin{equation}
d_{\kappa-GH}(X_1,X_3) \leq d_{\kappa-GH}(X_1,X_2) + d_{\kappa-GH}(X_2,X_3).
\end{equation}
\end{prop}

The following lemma, which simplifies the optimization procedure in the definition of $d_{\kappa-GH}$, will be used in the proof of 
Proposition .\ref{pop:triangle inequality for d_kappa-GH}.

\begin{lem}\label{lem:reduction in d_kappa-GH}
Let $X, X'$ be compact metric spaces and fix a countable dense set $\mathcal{N}\subset X$. Then 
\begin{equation}
d_{\kappa-GH}(X,X') = \inf d^{\ell^\infty}_{H}(\kappa_{\mathcal{N}}(X),\kappa_{\mathcal{N}'}(X')),
\end{equation}
where the infimum is over all countable and dense sets $\mathcal{N}'\subset X'$.
\end{lem}

\begin{proof}
It is clear that
\begin{equation}
d^{\mathcal{N}}_{\kappa-GH}(X,X') \geq d_{\kappa-GH}(X,X').
\end{equation}
On the other hand, given $\varepsilon>0$ there exist countable dense subsets 
\[
\mathcal{N}_1=\{q_i\}_{i\in\mathbb{N}}\subset X,  \qquad \mathcal{N}'_1=\{r_i\}_{i\in\mathbb{N}}\subset X'
\]
such that
\begin{equation}\label{eq-suitableN1}
    d^{\ell^\infty}_{H}(\kappa_{\mathcal{N}_1}(X),\kappa_{\mathcal{N}'_1}(X')) < d_{\kappa-GH}(X,X') + \varepsilon.
\end{equation}
Since $\mathcal N_1$ is countable and dense, one can 
define $\varphi\colon \mathbb{N}\to\mathbb{N}$ such that
\begin{equation}\label{eq:choice of points in dense set}
d_{X}(p_i,q_{\varphi(i)}) < \varepsilon,
\end{equation}
where $\mathcal{N} = \{p_i\}_{i\in\mathbb{N}}$.
Setting
\begin{align}
\mathcal{M} = \{q_{\varphi(i)}\}_{i\in\mathbb{N}} \subset \mathcal{N}_1 \label{eq:subset of dense set 1}\\
\mathcal{M}' = \{r_{\varphi(i)}\}_{i\in\mathbb{N}}\subset \mathcal{N}_1' \label{eq:subset of dense set 2},
\end{align}
we can define maps $\phi_{\mathcal{M}}\colon X\to \ell^\infty$ and $\phi_{\mathcal{M}'}\colon X'\to \ell^\infty$ given by 
\be
\phi_{\mathcal{M}}(x)=
(d_X(q_{\varphi(1)},x),d_X(q_{\varphi(2)},x),\ldots)
\ee
and 
\be
\phi_{\mathcal{M'}}(x')=
(d_{X'}(r_{\varphi(1)},x'),d_{X'}(r_{\varphi(2)},x'),\ldots).
\ee
The fact that $\mathcal{M}$ and $\mathcal{M}'$ are not necessarily dense implies that they are not necessarily isometric embeddings but this is irrelevant in the rest of the argument. 
By \eqref{eq:choice of points in dense set}, it follows that
\be
d^{\ell^\infty}_{H}(\kappa_{\mathcal{N}}(X),\phi_{\mathcal{M}}(X))  \leq \varepsilon
\ee
and by the definition of 
$d^{\ell^\infty}_{H}(\kappa_{\mathcal{N}_1}(X),\kappa_{\mathcal{N}'_1}(X'))$, \eqref{eq:subset of dense set 1} and \eqref{eq:subset of dense set 2},
\be
d^{\ell^\infty}_{H}(\phi_{\mathcal{M}}(X),\phi_{\mathcal{M}'}(X'))  \leq 
d^{\ell^\infty}_{H}(\kappa_{\mathcal{N}_1}(X),\kappa_{\mathcal{N}'_1}(X')).
\ee
Combining the previous inequalities, the triangle inequality for $d^{\ell^\infty}_{H}$ and \eqref{eq-suitableN1}: 
\begin{align}
d^{\ell^\infty}_{H}(\kappa_{\mathcal{N}}(X),\kappa_{\mathcal{N}'}(X')) &\leq d^{\ell^\infty}_{H}(\kappa_{\mathcal{N}}(X),\phi_{\mathcal{M}}(X)) + d^{\ell^\infty}_{H}(\phi_{\mathcal{M}}(X),\phi_{\mathcal{M}'}(X')) \\
&\leq \varepsilon + d^{\ell^\infty}_{H}(\kappa_{\mathcal{N}_1}(X),\kappa_{\mathcal{N}'_1}(X')) \\
&< 2\varepsilon + d_{\kappa-GH}(X,X').
\end{align}
This implies
\begin{equation}
    d^{\mathcal{N}}_{\kappa-GH}(X,X') < 2\varepsilon+d_{\kappa-GH}(X,X')
\end{equation}
and by letting $\varepsilon\to 0$, the claim follows.
\end{proof}

\begin{proof}[Proof of Proposition~\ref{pop:triangle inequality for d_kappa-GH}]
Define $d_{ij}=d_{\kappa-GH}(X_i,X_j)$ for $i,j\in\{1,2,3\}$. Let $\varepsilon>0$ and $\mathcal{N}_2 \subset X_2$ be a countable dense set. By Lemma~\ref{lem:reduction in d_kappa-GH}, there exist $\mathcal{N}_1\subset X_1$ and  $\mathcal{N}_3\subset X_3$ countable and dense sets such that 
\begin{align}
&d^{\ell^\infty}_{H}(\kappa_{\mathcal{N}_1}(X_1), \kappa_{\mathcal{N}_2}(X_2)) < d_{12}+\varepsilon,\\
&d^{\ell^\infty}_{H}(\kappa_{\mathcal{N}_2}(X_2), \kappa_{\mathcal{N}_3}(X_3)) < d_{23}+\varepsilon.
\end{align}
By definition of $d_{\kappa-GH}$ and the triangle inequality for $d^{\ell^\infty}_{H}$:
\begin{align}
d_{13} &\leq d^{\ell^\infty}_{H}(\kappa_{\mathcal{N}_1}(X_1), \kappa_{\mathcal{N}_3}(X_3)) \\
&\leq d^{\ell^\infty}_{H}(\kappa_{\mathcal{N}_1}(X_1), \kappa_{\mathcal{N}_2}(X_2))  + d^{\ell^\infty}_{H}(\kappa_{\mathcal{N}_2}(X_2), \kappa_{\mathcal{N}_3}(X_3))\\ 
&<  d_{12}+d_{23} + 2\varepsilon.
\end{align}
By letting $\varepsilon\to 0$, the result follows.
\end{proof}

\begin{proof}[Proof of Proposition \ref{prop:triangle inequality for d_tau-H}]
Lemma \ref{lem:reduction in d_kappa-GH} and its proof
can be adapted to timed-metric spaces 
by adding as first coordinate the corresponding time function 
to all the maps with target $\ell^\infty$.
 With this at hand, 
  the proof of Proposition \ref{pop:triangle inequality for d_kappa-GH} 
 can be adapted to timed-metric spaces as above. 
 This finishes the proof. 
\end{proof}

\bibliographystyle{plainurl}
\bibliography{CPS-bib}

\begin{thebibliography}{10}

\bibitem{Che-Gomez-pairs}
Andr\'es Ahumada~G\'omez and Mauricio Che.
\newblock Gromov-{H}ausdorff convergence of metric pairs and metric tuples.
\newblock {\em Differential Geom. Appl.}, 94:Paper No. 102135, 29, 2024.
\newblock \href {https://doi.org/10.1016/j.difgeo.2024.102135} {\path{doi:10.1016/j.difgeo.2024.102135}}.

\bibitem{Allen-Null}
Brian Allen.
\newblock Null distance and {G}romov-{H}ausdorff convergence of warped product spacetimes.
\newblock {\em Gen. Relativity Gravitation}, 55(10):Paper No. 118, 34, 2023.
\newblock \href {https://doi.org/10.1007/s10714-023-03167-8} {\path{doi:10.1007/s10714-023-03167-8}}.

\bibitem{Allen-Burtscher-22}
Brian Allen and Annegret Burtscher.
\newblock Properties of the null distance and spacetime convergence.
\newblock {\em Int. Math. Res. Not. IMRN}, 10:7729--7808, 2022.
\newblock \href {https://doi.org/10.1093/imrn/rnaa311} {\path{doi:10.1093/imrn/rnaa311}}.

\bibitem{AGH}
Lars Andersson, Gregory~J. Galloway, and Ralph Howard.
\newblock The cosmological time function.
\newblock {\em Classical Quantum Gravity}, 15(2):309--322, 1998.
\newblock URL: \url{http://dx.doi.org.memex.lehman.cuny.edu:2048/10.1088/0264-9381/15/2/006}, \href {https://doi.org/10.1088/0264-9381/15/2/006} {\path{doi:10.1088/0264-9381/15/2/006}}.

\bibitem{BBI}
Dmitri Burago, Yuri Burago, and Sergei Ivanov.
\newblock {\em A course in metric geometry}, volume~33 of {\em Graduate Studies in Mathematics}.
\newblock American Mathematical Society, Providence, RI, 2001.

\bibitem{Burtscher-Garcia-Heveling-Global}
Annegret Burtscher and Leonardo Garc\'ia-Heveling.
\newblock Global hyperbolicity through the eyes of the null distance.
\newblock {\em Comm. Math. Phys.}, 405(4):Paper No. 90, 35, 2024.
\newblock \href {https://doi.org/10.1007/s00220-024-04936-5} {\path{doi:10.1007/s00220-024-04936-5}}.

\bibitem{CGGGMS}
Mauricio Che, Fernando Galaz-Garc\'ia, Luis Guijarro, and Ingrid~Amaranta Membrillo~Solis.
\newblock Metric geometry of spaces of persistence diagrams.
\newblock {\em J. Appl. Comput. Topol.}, 8(8):2197--2246, 2024.
\newblock \href {https://doi.org/10.1007/s41468-024-00189-2} {\path{doi:10.1007/s41468-024-00189-2}}.

\bibitem{ChCo-almost-rigidity}
Jeff Cheeger and Tobias~H. Colding.
\newblock Lower bounds on {R}icci curvature and the almost rigidity of warped products.
\newblock {\em Ann. of Math. (2)}, 144(1):189--237, 1996.
\newblock URL: \url{http://dx.doi.org/10.2307/2118589}, \href {https://doi.org/10.2307/2118589} {\path{doi:10.2307/2118589}}.

\bibitem{ChCo-PartI}
Jeff Cheeger and Tobias~H. Colding.
\newblock On the structure of spaces with {R}icci curvature bounded below. {I}.
\newblock {\em J. Differential Geom.}, 46(3):406--480, 1997.

\bibitem{CheegerGromovI}
Jeff Cheeger and Mikhael Gromov.
\newblock Collapsing {R}iemannian manifolds while keeping their curvature bounded. {I}.
\newblock {\em J. Differential Geom.}, 23(3):309--346, 1986.
\newblock URL: \url{http://projecteuclid.org/getRecord?id=euclid.jdg/1214440117}.

\bibitem{Cheeger-Naber-Invent-2013}
Jeff Cheeger and Aaron Naber.
\newblock Lower bounds on {R}icci curvature and quantitative behavior of singular sets.
\newblock {\em Invent. Math.}, 191(2):321--339, 2013.
\newblock URL: \url{http://dx.doi.org.memex.lehman.cuny.edu:2048/10.1007/s00222-012-0394-3}, \href {https://doi.org/10.1007/s00222-012-0394-3} {\path{doi:10.1007/s00222-012-0394-3}}.

\bibitem{Frechet1910}
Maurice Fr\'echet.
\newblock Les dimensions d'un ensemble abstrait.
\newblock {\em Math. Ann.}, 68(2):145--168, 1910.
\newblock \href {https://doi.org/10.1007/BF01474158} {\path{doi:10.1007/BF01474158}}.

\bibitem{Frechet}
Maurice Fréchet.
\newblock L’expression la plus générale de la ‘distance’ sur une droite.
\newblock {\em American Journal of Mathematics}, 47(1):1--10, 1925.

\bibitem{Fukaya-87}
Kenji Fukaya.
\newblock Collapsing of {R}iemannian manifolds and eigenvalues of {L}aplace operator.
\newblock {\em Invent. Math.}, 87(3):517--547, 1987.

\bibitem{Greene-Petersen}
Robert~E. Greene and Peter Petersen~V.
\newblock Little topology, big volume.
\newblock {\em Duke Math. J.}, 67(2):273--290, 1992.

\bibitem{Gromov-poly}
Mikhael Gromov.
\newblock Groups of polynomial growth and expanding maps.
\newblock {\em Inst. Hautes \'Etudes Sci. Publ. Math.}, 53:53--73, 1981.
\newblock URL: \url{http://www.numdam.org/item?id=PMIHES_1981__53__53_0}.

\bibitem{Gromov-1981}
Mikhael Gromov.
\newblock {\em Structures m\'etriques pour les vari\'et\'es riemanniennes}, volume~1 of {\em Textes Math\'ematiques [Mathematical Texts]}.
\newblock CEDIC, Paris, 1981.
\newblock Edited by J. Lafontaine and P. Pansu.

\bibitem{Gromov-metric}
Misha Gromov.
\newblock {\em Metric structures for {R}iemannian and non-{R}iemannian spaces}, volume 152 of {\em Progress in Mathematics}.
\newblock Birkh\"auser Boston Inc., Boston, MA, 1981.
\newblock 1999, Based on the 1981 French original [ MR0682063 (85e:53051)], With appendices by M. Katz, P. Pansu and S. Semmes, Translated from the French by Sean Michael Bates.

\bibitem{Kunzinger-Saemann}
Michael Kunzinger and Clemens S\"amann.
\newblock {L}orentzian length spaces.
\newblock {\em Ann. Global Anal. Geom.}, 54(3):399--447, 2018.
\newblock \href {https://doi.org/10.1007/s10455-018-9633-1} {\path{doi:10.1007/s10455-018-9633-1}}.

\bibitem{Kunzinger-Steinbauer-22}
Michael Kunzinger and Roland Steinbauer.
\newblock Null distance and convergence of {L}orentzian length spaces.
\newblock {\em Ann. Henri Poincar\'e}, 23(12):4319--4342, 2022.
\newblock \href {https://doi.org/10.1007/s00023-022-01198-6} {\path{doi:10.1007/s00023-022-01198-6}}.

\bibitem{Kuratowski1935}
Casimir Kuratowski.
\newblock Quelques problèmes concernant les espaces métriques non-séparables.
\newblock {\em Fundamenta Mathematicae}, 25(1):534--545, 1935.
\newblock URL: \url{http://eudml.org/doc/212809}.

\bibitem{Minguzzi-Suhr-24}
Ettore Minguzzi and Stefan Suhr.
\newblock {L}orentzian metric spaces and their {G}romov-{H}ausdorff convergence.
\newblock {\em Lett. Math. Phys.}, 114(3):Paper No. 73, 63, 2024.
\newblock \href {https://doi.org/10.1007/s11005-024-01813-z} {\path{doi:10.1007/s11005-024-01813-z}}.

\bibitem{Mondino-Saemann-2025}
Andrea Mondino and Clemens Sämann.
\newblock {Lorentzian Gromov-Hausdorff convergence and pre-compactness}, 2025.
\newblock URL: \url{https://arxiv.org/abs/2504.10380}, \href {https://arxiv.org/abs/2504.10380} {\path{arXiv:2504.10380}}.

\bibitem{Noldus-limit}
Johan Noldus.
\newblock The limit space of a {C}auchy sequence of globally hyperbolic spacetimes.
\newblock {\em Classical Quantum Gravity}, 21(4):851--874, 2004.
\newblock URL: \url{http://dx.doi.org/10.1088/0264-9381/21/4/008}, \href {https://doi.org/10.1088/0264-9381/21/4/008} {\path{doi:10.1088/0264-9381/21/4/008}}.

\bibitem{Noldus}
Johan Noldus.
\newblock A {L}orentzian {G}romov-{H}ausdorff notion of distance.
\newblock {\em Classical Quantum Gravity}, 21(4):839--850, 2004.
\newblock URL: \url{http://dx.doi.org/10.1088/0264-9381/21/4/007}, \href {https://doi.org/10.1088/0264-9381/21/4/007} {\path{doi:10.1088/0264-9381/21/4/007}}.

\bibitem{SakSor-Notions}
Anna Sakovich and Christina Sormani.
\newblock Introducing various notions of distances between space-times, 2025.
\newblock URL: \url{https://arxiv.org/abs/2410.16800}, \href {https://arxiv.org/abs/2410.16800} {\path{arXiv:2410.16800}}.

\bibitem{SakSor-SIF}
Anna Sakovich and Christina Sormani.
\newblock Space-time intrinsic flat convergence.
\newblock {\em in progress}, 2025.

\bibitem{Sormani-ArzAsc}
Christina Sormani.
\newblock Intrinsic flat {A}rzela-{A}scoli theorems.
\newblock {\em Comm. Anal. Geom.}, 26(6):1317--1373, 2018.
\newblock \href {https://doi.org/10.4310/CAG.2018.v26.n6.a3} {\path{doi:10.4310/CAG.2018.v26.n6.a3}}.

\bibitem{SV-BigBang}
Christina Sormani and Carlos Vega.
\newblock Big bang spacetimes.
\newblock {\em unpublished ideas}, 2014.

\bibitem{SV-Null}
Christina Sormani and Carlos Vega.
\newblock Null distance on a spacetime.
\newblock {\em Classical Quantum Gravity}, 33(8):085001, 29, 2016.
\newblock URL: \url{http://dx.doi.org/10.1088/0264-9381/33/7/085001}, \href {https://doi.org/10.1088/0264-9381/33/7/085001} {\path{doi:10.1088/0264-9381/33/7/085001}}.

\bibitem{SorWei3}
Christina Sormani and Guofang Wei.
\newblock The covering spectrum of a compact length space.
\newblock {\em J. Differential Geom.}, 67(1):35--77, 2004.

\end{thebibliography}
\end{document}